\documentclass[oneside,11pt]{amsart}
\usepackage{amsmath,amsfonts, latexsym,amssymb}
\usepackage[active]{srcltx}
\newtheorem{theorem}{Theorem}[section]
\newtheorem{lemma}[theorem]{Lemma}

\theoremstyle{definition}

\newtheorem{remark}[theorem]{Remark}
\numberwithin{equation}{section}


\begin{document}

{\renewcommand{\thefootnote}{}
\setcounter{footnote}{0}}

\title{An elementary proof of a congruence by Skula and Granville}

\author{Romeo Me\v strovi\' c}

\address{Maritime Faculty, University of Montenegro, Dobrota 36,
 85330 Kotor, Montenegro} \email{romeo@ac.me}

\maketitle

   \begin{abstract} 
Let $p\ge 5$ be a  prime, 
and let $q_p(2):=(2^{p-1}-1)/p$ be the Fermat quotient of $p$
to base $2$. 
The following curious congruence  was conjectured by L. Skula 
and proved by A. Granville
     $$
q_p(2)^2\equiv -\sum_{k=1}^{p-1}\frac{2^k}{k^2}\pmod{p}.
  $$
In this note we establish the above congruence by entirely elementary
number theory arguments.
   \end{abstract}

\section{Introduction and Statement  of  the Main Result}

  The Fermat Little Theorem states that if $p$ is a 
prime and $a$ is an integer not divisible by $p$,
then $a^{p-1}\equiv 1\,(\bmod{\,\,p})$. This gives rise
to the definition of the {\it Fermat quotient} of $p$ to base $a$
   $$
q_p(a):=\frac{a^{p-1}-1}{p},
    $$  
which is an integer. 
Fermat quotients played an important role in the 
study of cyclotomic fields and Fermat Last Theorem.
More precisely, divisibility of Fermat quotient $q_p(a)$ by $p$
has numerous applications which include the Fermat Last Theorem and
squarefreeness testing (see \cite{as1}, \cite{cp},
\cite{em}, \cite{gr1}  and \cite{l}). 
Ribenboim \cite{r} and  Granville \cite{gr1}, 
besides proving new results, provide a review of known facts and open problems.
 
By a classical Glaisher's result
(see \cite{gl} or \cite{gr}) for a prime $p\ge 3$,
  \begin{equation}\label{con1.1}
q_p(2)\equiv -\frac{1}{2}\sum_{k=1}^{p-1}\frac{2^k}{k}\pmod{p}.
  \end{equation}
Recently Skula conjectured that for any prime $p\ge 5$,
  \begin{equation}\label{con1.2}
q_p(2)^2\equiv -\sum_{k=1}^{p-1}\frac{2^k}{k^2}\pmod{p}.
  \end{equation}
Applying  certain polynomial congruences, 
Granville \cite{gr} proved the congruence (\ref{con1.2}).
In this note, we give an elementary proof of 
this congruence which is based on congruences for some  harmonic type sums.

 \begin{remark}
Recently, given a prime $p$ and a positive integer $r<p-1$,
R. Tauraso \cite[Theorem 2.3]{t} established the 
congruence $\sum_{k=1}^{p-1}2^k/k^r\,(\bmod{\,\,p})$ in terms 
of an alternating $r$-tiple harmonic sum.  For example, combining this 
result when $r=2$ with  the congruence (\ref{con1.2})
\cite[Corollary 2.4]{t}, it follows that 
  $$ 
\sum_{1\le i<j\le p-1}\frac{(-1)^j}{ij}
\equiv q_p(2)^2\equiv -\sum_{k=1}^{p-1}\frac{2^k}{k^2}\pmod{p}.
  $$
\end{remark}

\section{Proof of the congruence (1.2)}

The {\it harmonic numbers} $H_n$ are defined by
   $$
H_n:=\sum_{j=1}^n\frac{1}{j},\quad n=1,2,\ldots,
  $$
where by convention $H_0=0$.

\begin{lemma} For any  prime $p\ge 5$ we have 
    $$
q_p(2)^2\equiv \sum_{k=1}^{p-1}\left(2^k+\frac{1}{2^k}\right)\frac{H_k}{k+1}
\pmod{p}.
\leqno(2.1)
   $$
\end{lemma}

\begin{proof} 
In the present proof we will always suppose
that $i$ and $j$ are positive integers such that $i\le p-1$ and $j\le p-1$,
and that all the summations including $i$ and $j$ range over the set 
of such pairs $(i,j)$. 

Using the congruence (1.1) and the fact that
by Fermat Little Theorem, $2^{p-1}\equiv 1\,(\bmod{\,\, p})$,
we get 
  \begin{eqnarray*}
q_p(2)^2&=&\left(\frac{2^{p-1}-1}{p}\right)^2\equiv
\frac{1}{4}\left(\sum_{k=1}^{p-1}\frac{2^{k}}{k}\right)^2=
\frac{1}{4}\left(\sum_{k=1}^{p-1}\frac{2^{p-k}}{p-k}\right)^2\\
(2.2)\qquad\qquad\qquad &\equiv& \frac{1}{4}\left(2\sum_{k=1}^{p-1}
\frac{2^{(p-1)-k}}{-k}\right)^2
\equiv\left(\sum_{k=1}^{p-1}\frac{1}{k\cdot 2^k}\right)^2\qquad\qquad
\qquad\\
&=& \sum_{i+j\le p}\frac{1}{ij\cdot 2^{i+j}}+
\sum_{i+j\ge p}
\frac{1}{ij\cdot 2^{i+j}}-
\sum_{i+j= p}\frac{1}{ij\cdot 2^{i+j}}\\
&:=& S_1+S_2-S_3\pmod{p}.
  \end{eqnarray*}
We will determine $S_1$, $S_2$ and $S_3$ modulo $p$ as follows.
    \begin{eqnarray*}
S_1 &=& \sum_{i+j\le p}\frac{1}{ij\cdot 2^{i+j}}
=\sum_{k=2}^p\sum_{i+j= k}\frac{1}{ij\cdot 2^{k}}\\
(2.3)\qquad &=&\sum_{k=2}^p\frac{1}{2^k}\cdot\frac{1}{k}
\sum_{i= 1}^{k-1}\left(\frac{1}{i}+\frac{1}{k-i}\right)=
\sum_{k=2}^p\frac{2H_{k-1}}{k\cdot 2^k}=
\sum_{k=1}^{p-1}\frac{H_{k}}{(k+1) 2^{k}}.\qquad
   \end{eqnarray*}
Observe that the pair $(i,j)$ satisfies $i+j=k$
for some $k\in\{p,p+1,\ldots,2p-2\}$ if and only if
for such a $k$ holds $(p-i)+(p-j)=l$ with     
 $l:=2p-k\le p$. Accordingly, 
using the fact that by Fermat Little Theorem, 
$2^{2p}\equiv 2^2\,(\bmod{\,\, p})$, we have  
   \begin{eqnarray*}
S_2 &=&\sum_{i+j\ge p}\frac{1}{ij\cdot 2^{i+j}}
=\sum_{(p-i)+(p-j)\ge p}^{p-1}\frac{1}{(p-i)(p-j)\cdot 2^{(p-i)+(p-j)}}\\
 &\equiv&\sum_{i+j\le p}\frac{1}{ij\cdot 2^{2p-(i+j)}}
\equiv\frac{1}{4}\sum_{i+j\le p}\frac{2^{i+j}}{ij}=
\frac{1}{4}\sum_{k=2}^p\sum_{i+j= k}\frac{2^k}{ij}\\
(2.4)\qquad\qquad&=&\frac{1}{4}\sum_{k=2}^p\frac{2^k}{k}
\sum_{i= 1}^{k-1}\left(\frac{1}{i}+\frac{1}{k-i}\right)=
\sum_{k=2}^{p}\frac{2^{k-1}H_{k-1}}{k}\qquad\qquad\qquad\qquad\\
&=& \sum_{k=1}^{p-1}\frac{2^{k}H_{k}}{k+1}\pmod{p}.
   \end{eqnarray*}
By {\it Wolstenholme's theorem} (see, e.g., \cite{w}, 
\cite{gr0}; for its generalizations see \cite[Theorems 1 and 2]{sl})
if $p$ is a prime greater than 3, then the numerator of the 
fraction $H_{p-1}=1+\frac{1}{2}+\frac{1}{3}+\cdots+\frac{1}{p-1}$ is 
divisible by $p^2$. Hence, we find that
  \begin{eqnarray*}
S_3&=&\sum_{i+j= p}\frac{2^{i+j}}{ij}=2^p\sum_{i=1}^{p-1}
\frac{1}{i(p-i)}\\
(2.5)\qquad \qquad
&=& \frac{2^{p}}{p}\sum_{i=1}^{p-1}\left(\frac{1}{i}+\frac{1}{p-i}\right)=
\frac{2^{p+1}}{p}H_{p-1}\equiv 0\pmod{p}.
\qquad \qquad
  \end{eqnarray*}
Finally, substituting (2.3), (2.4) and (2.5) into (2.2), we 
immediately obtain (2.1).
\end{proof}

Proof of the following result easily follows from 
the congruence $H_{p-1}\equiv 0\,(\bmod{\,\, p})$.

\begin{lemma}\cite[Lemma 2.1]{s} Let $p$ be an odd prime. Then 
  $$
H_{p-k-1}\equiv H_{k}\pmod{p}\leqno(2.6)
  $$
for every $k=1,2,\ldots p-2$.
\end{lemma}

\begin{lemma} For any prime $p\ge 5$ we have
  $$
q_p(2)^2\equiv \sum_{k=1}^{p-1}\frac{H_k}{k\cdot 2^k}-
\sum_{k=1}^{p-1}\frac{2^k}{k^2}\pmod{p}.\leqno(2.7)
  $$
\end{lemma}
\begin{proof}
Since by Wolstenholme's theorem,
$H_{p-1}/p\equiv 0\,(\bmod{\,\, p})$, 
using this and the congruences 
$2^{p-1}\equiv 1\,(\bmod{\,\, p})$ and (2.6) of Lemma 2.2, 
we immediately obtain
   \begin{eqnarray*}
\sum_{k=1}^{p-1}\frac{2^kH_k}{k+1}
&\equiv& \sum_{k=1}^{p-2}\frac{2^kH_k}{k+1}
= \sum_{k=1}^{p-2}\frac{2^{p-k-1}H_{p-k-1}}{p-k}\\
(2.8)\qquad\qquad\qquad &\equiv&-\sum_{k=1}^{p-2}\frac{H_k}{k\cdot 2^k}
\equiv-\sum_{k=1}^{p-1}\frac{H_k}{k\cdot 2^k}\pmod{p}.
\qquad\qquad\qquad
    \end{eqnarray*}
Further, we have
  \begin{eqnarray*}
\sum_{k=1}^{p-2}\frac{H_k}{(k+1)2^k}
&=&2\sum_{k=1}^{p-2}\frac{H_{k+1}-\frac{1}{k+1}}{(k+1)2^{k+1}}\\
(2.9)\qquad\qquad\qquad\qquad\qquad &=& 2\sum_{k=1}^{p-1}
\frac{H_{k}}{k\cdot 2^{k}}-
2\sum_{k=1}^{p-1}\frac{1}{k^2\cdot 2^k}.\qquad\qquad\qquad\qquad\qquad
   \end{eqnarray*}
Moreover, from $2^{p}\equiv 2\,(\bmod{\,\, p})$ we have 
  \begin{eqnarray*}
\sum_{k=1}^{p-1}\frac{1}{k^2\cdot 2^k}
&=&\sum_{k=1}^{p-1}\frac{1}{(p-k)^2\cdot 2^{p-k}}\\
(2.10)\qquad\qquad\qquad\qquad\qquad &\equiv
&\sum_{k=1}^{p-1}\frac{1}{k^2\cdot 2^{1-k}}=
\frac{1}{2}\sum_{k=1}^{p-1}\frac{2^k}{k^2}\pmod{p}.\qquad\qquad\qquad\qquad
   \end{eqnarray*}
The congruences (2.8), (2.9) and (2.10) immediately yield
  \begin{eqnarray*}
 \sum_{k=1}^{p-1}\left(2^k+\frac{1}{2^k}\right)\frac{H_k}{k+1}
&=&\sum_{k=1}^{p-1}\frac{2^kH_k}{k+1}+\sum_{k=1}^{p-1}\frac{H_k}{(k+1)2^k}\\
(2.11)\qquad\qquad\qquad\qquad &\equiv& \sum_{k=1}^{p-1}\frac{H_k}{k\cdot 2^k}-
\sum_{k=1}^{p-1}\frac{2^k}{k^2}\pmod{p}.
\qquad\qquad\qquad\qquad
    \end{eqnarray*}
Finally, comparing (2.1) of Lemma 2.1 with (2.11), we obtain the desired congruence
(2.7).
   \end{proof}
Notice that the congruence 
$\sum_{k=1}^{p-1}\frac{H_k}{k\cdot 2^k}\equiv 0\,(\bmod{\,\,p})$ 
with a prime $p\ge 5$ is recently established by Z.W. 
Sun \cite[Theorem 1.1 (1.1)]{s} and it is based on the identity
from  \cite[Lemma 2.4]{s}.
Here we give another simple proof of this congruence (Lemma 2.6).
 \begin{lemma} For any prime $p\ge 5$ we have 
  $$
\sum_{k=1}^{p-1}\frac{H_k}{k\cdot 2^k}\equiv 
\frac{1}{2}\sum_{1\le i\le j\le p-1}\frac{2^i-1}{ij}\pmod{p}.\leqno{(2.12)}
   $$   
\end{lemma}
\begin{proof} From the identity 
  $$
  \left(\sum_{k=1}^{p-1}\frac{1}{k}\right)
\left(\sum_{k=1}^{p-1}\frac{1}{k\cdot 2^k}\right)=
\sum_{1\le i<j\le p-1}\frac{1}{ij\cdot 2^j}+
\sum_{1\le j<i\le p-1}\frac{1}{ij\cdot 2^j}+
\sum_{k=1}^{p-1}\frac{1}{k^2\cdot 2^k},
  $$
and the congruence 
$H_{p-1}=1+\frac{1}{2}+\frac{1}{3}+\cdots+\frac{1}{p-1}\equiv 0\,
(\bmod\,\, p)$ it follows that 
  $$
\sum_{1\le i<j\le p-1}\frac{1}{ij\cdot 2^j}+
\sum_{1\le j<i\le p-1}\frac{1}{ij\cdot 2^j}+
\sum_{k=1}^{p-1}\frac{1}{k^2\cdot 2^k}\equiv 0\pmod{p}.\leqno{(2.13)}
  $$
Since  $2^{p}\equiv 2\,(\bmod{\,\,p})$, we have 
  $$
\sum_{1\le j<i\le p-1}\frac{1}{ij\cdot 2^j}\equiv
\sum_{1\le j<i\le p-1}\frac{1}{2}\frac{2^{p-j}}{(p-i)(p-j)}
\equiv \frac{1}{2}\sum_{1\le i<j\le p-1}\frac{2^j}{ij}\pmod{p},
  $$
which substituting into (2.13) gives 
  $$
\sum_{1\le i<j\le p-1}\frac{1}{ij\cdot 2^j}+
\sum_{k=1}^{p-1}\frac{1}{k^2\cdot 2^k}\equiv
-\frac{1}{2}\sum_{1\le i<j\le p-1}\frac{2^j}{ij}
\pmod{p}.\leqno{(2.14)}
  $$
Further, if  we observe that
  $$
\sum_{k=1}^{p-1}\frac{H_k}{k\cdot 2^k}=
\sum_{k=1}^{p-1}\frac{H_{k-1}+\frac{1}{k}}{k\cdot 2^k}=
\sum_{1\le i< j\le p-1}\frac{1}{ij\cdot 2^j}
+\sum_{k=1}^{p-1}\frac{1}{k^2\cdot 2^k},
   $$  
then substituting (2.14) into the previous identity,
we obtain 
 $$
\sum_{k=1}^{p-1}\frac{H_k}{k\cdot 2^k}\equiv 
-\frac{1}{2}\sum_{1\le i<j\le p-1}\frac{2^j}{ij}\pmod{p}.\leqno{(2.15)}
   $$   
Since 
  $$
0\equiv \left(\sum_{k=1}^{p-1}\frac{1}{k}\right)
\left(\sum_{k=1}^{p-1}\frac{2^k}{k}\right)=
\sum_{1\le j\le i\le p-1}\frac{2^j}{ij}+
\sum_{1\le i<j\le p-1}\frac{2^j}{ij}\pmod{p},
   $$
comparing this with (2.15), we immediately obtain 
 $$
\sum_{k=1}^{p-1}\frac{H_k}{k\cdot 2^k}\equiv 
\frac{1}{2}\sum_{1\le i\le j\le p-1}\frac{2^i}{ij}\pmod{p}.\leqno{(2.16)}
   $$
From a well known fact that (see e.g., \cite[p. 353]{l})
 $$
\sum_{k=1}^{p-1}\frac{1}{k^2}\equiv 0\pmod{p}\leqno{(2.17)}
  $$
we find that 
    $$
  \sum_{1\le i\le j\le p-1}\frac{1}{ij}=
\frac{1}{2}\left(\left(\sum_{k=1}^{p-1}\frac{1}{k}\right)^2+
\sum_{k=1}^{p-1}\frac{1}{k^2}\right)\equiv 0\pmod{p}.
   $$
Finally, the above congruence and (2.16) immediately 
yield the desired congruence (2.12).  
   \end{proof}

\begin{lemma} For any positive integer $n$ holds 
       $$
\sum_{1\le i\le j\le n}\frac{2^i-1}{ij}
=\sum_{k=1}^{n}\frac{1}{k^2}{n\choose k}.\leqno{(2.18)}
       $$   
\end{lemma}
\begin{proof} 
Using the well known identities 
$\sum_{i=k}^{j}{i-1\choose k-1}={j\choose k}$ 
and $\frac{1}{j}{j\choose k}=\frac{1}{k}{j-1\choose k-1}$ with 
$k\le j$, and the fact that ${i\choose k}=0$ when $i<k$, we have 
  \begin{eqnarray*}
\sum_{1\le i\le j\le n}\frac{2^i-1}{ij}
&=&\sum_{1\le i\le j\le n}\frac{(1+1)^i-1}{ij}=
\sum_{1\le i\le j\le n}\frac{1}{j}\sum_{k=1}^{i}\frac{1}{i}{i\choose k}\\
&=&\sum_{1\le i\le j\le n}\frac{1}{j}\sum_{k=1}^{n}
\frac{1}{k}{i-1\choose k-1}=\sum_{k=1}^{n}\frac{1}{k}
\sum_{1\le i\le j\le n}\frac{1}{j}{i-1\choose k-1}\\
&=&\sum_{k=1}^{n}\frac{1}{k}\sum_{k\le i\le j\le n}
\frac{1}{j}{i-1\choose k-1}=\sum_{k=1}^{n}\frac{1}{k}\sum_{j=i}^n\frac{1}{j}
\sum_{i=k}^j{i-1\choose k-1}\\
&=&\sum_{k=1}^{n}\frac{1}{k}\sum_{j=i}^n\frac{1}{j}{j\choose k}
=\sum_{k=1}^{n}\frac{1}{k}\sum_{j=k}^n\frac{1}{k}{j-1\choose k-1}\\
&=&\sum_{k=1}^{n}\frac{1}{k^2}\sum_{j=k}^n{j-1\choose k-1}=
\sum_{k=1}^{n}\frac{1}{k^2}{n\choose k},
  \end{eqnarray*}
as desired.
 \end{proof}

\begin{lemma}\cite[Theorem 1.1 (1.1)]{s}
For any prime $p\ge 5$ holds
  $$
\sum_{k=1}^{p-1}\frac{H_k}{k\cdot 2^k}\equiv 
0\pmod{p}.\leqno{(2.19)}
   $$   
\end{lemma}

\begin{proof}

Using the congruence (2.12) from Lemma 2.4 and the  identity 
(2.18) with $n=p-1$ in Lemma 2.5, we find that 
   $$
\sum_{k=1}^{p-1}\frac{H_k}{k\cdot 2^k}\equiv 
\sum_{k=1}^{p-1}\frac{1}{k^2}{p-1\choose k}\pmod{p}.\leqno{(2.20)}
   $$   
It is well known (see e.g., \cite{hw}) that for $k=1,2,\ldots,p-1,$
  $$
{p-1\choose k}\equiv (-1)^k\pmod{p}.\leqno{(2.21)}
  $$
Then from  (2.20), (2.21) and (2.17) we get  
 \begin{eqnarray*}
\sum_{k=1}^{p-1}\frac{H_k}{k\cdot 2^k}
&\equiv& \sum_{k=1}^{p-1}\frac{(-1)^k}{k^2}=
\sum_{k=1}^{p-1}\frac{1}{k^2}-
2\sum_{1\le j\le p-1\atop 2\mid j}\frac{1}{j^2}\\
&\equiv &
-2\sum_{1\le j\le p-1\atop 2\mid j}\frac{1}{j^2}= -\frac{1}{2}\sum_{k=1}^{(p-1)/2}\frac{1}{k^2}\pmod{p}.
 \end{eqnarray*}
Finally, the above congruence together with a well known 
fact that (see e.g., \cite[Corollary 5.2 (a) with $k=2$]{s1})
  $$
\sum_{k=1}^{(p-1)/2}\frac{1}{k^2}\equiv 0\pmod{p}
  $$
yields 
 $$
\sum_{k=1}^{p-1}\frac{H_k}{k\cdot 2^k}\equiv 0\pmod{p}.
  $$
This concludes the proof.
\end{proof}
\begin{proof}[Proof of the congruence $(1.2)$]
The congruence (1.2) immediately follows from (2.7) of Lemma 2.3
and (2.19) of Lemma 2.6.
\end{proof}


\begin{thebibliography}{99}
\bibitem{as1} Agoh, T., Skula, L., {\it Fermat quotients for composite moduli},
 J. Number Theory 66 (1997) 29--50.

\bibitem{cp}  Cao, H. Q.,  Pan, H., {\it A congruence involving product
of $q$-binomial coefficientc}, J. Number Theory 
{\bf 121} (2006), 224--233.

\bibitem{em}   Ernvall. R., 
 Mets\"{a}nkyl\"{a}, T., {\it On the $p$-divisibylity of Fermat quotients},
 Math. Comp. {\bf 66} (1997), 1353--1365. 

\bibitem{gl}  Glaisher, J. W. L., 
{\it On the residues of the sums of the inverse 
powers of numbers in arithmetical progression}, Q.  J. Math. {\bf 32} 
(1900), 271-288. 

\bibitem{gr}  Granville, A., {\it The square of the Fermat quotient},
 Integers {\bf 4} (2004), \# A22.

\bibitem{gr0}  Granville, A.,  
{\it Arithmetic properties of binomial coefficients. $I$.
Binomial coefficients modulo prime powers}, in 
Organic Mathematics--Burnaby, BC 1995,  CMS Conf. Proc., vol. 20, 
American  Mathematical Society, Providence, RI, 1997, 253-276.

  \bibitem{gr1}  Granville, A.,  {\it Some conjectures related to 
Fermat's Last Theorem}, Number Theory (Banff, AB, 1988),
 de Gruyter, Berlin, 1990, 177--192.

\bibitem{hw}  Hardy, G. H.,   Wright, E. M.,
{\it An Introduction to the Theory of Numbers}, Fourth Edition, 
Clarendon Press, Oxford, 1960.

\bibitem{l}  Lehmer, E., {\it On congruences 
involving Bernoulli numbers and the quotients of
Fermat and Wilson},  Ann. Math. {\bf 39} (1938), 350--360. 

\bibitem{r}  Ribenboim, P., 13  {\it Lectures on Fermat's Last Theorem}, 
Springer-Verlag, New York, Heidelberg, Berlin, 1979.


\bibitem{sl}  Slavutsky, I. Sh., {\it Leudesdorf's theorem and Bernoulli
numbers},  Arch. Math. {\bf 35} (1999), 299--303.

\bibitem{s1}  Sun, Z. H., {\it Congruences concerning Bernoulli numbers
and Bernoulli polynomials},  Discrete Appl. Math. {\bf 105} (2000),
193--223.


\bibitem{s}  Sun, Z. W., {\it Arithmetic theory of
harmonic  numbers}, Proc. Amer. Math. Soc., article in press; preprint 
{\tt arXiv:0911.4433v3 [math.NT]} (2009).


\bibitem{t}   Tauraso, R.,  {\it Congruences involving alternating multiple
harmonic sums}, Electron. J. Comb. {\bf 17} (2010),  \# R16.


\bibitem{w}  Wolstenholme, J., {\it On certain properties
of prime numbers},  Quart. J. Pure Appl. Math. {\bf 5} (1862), 35-39. 

\end{thebibliography}
\end{document}